\documentclass[english,a4paper,reqno,11pt]{amsart}
\usepackage[T1]{fontenc}
\usepackage{
    babel,
    lmodern,
    mathtools,
    amssymb,
    enumerate,
    amsrefs,
    amsthm,
    thmtools,
    xcolor,
    hyperref,
    cleveref
}
\usepackage[all]{xy}

\newtheorem{teo}{Theorem}[section]
\newtheorem{prop}[teo]{Proposition}
\newtheorem{lema}[teo]{Lemma}

\theoremstyle{definition}
	\newtheorem{dfn}[teo]{Definition}
	\newtheorem{obs}[teo]{Remark}
	
	\newtheorem{ex}[teo]{Example}

\newcommand{\Sp}{\mathrel{\underline{\#}}}
\newcommand{\Tr}{\mathrel{\underline{\otimes}}}
\newcommand{\ud}[1]{_{_{#1}}}

\newcommand{\Hp}{{\ensuremath{{}^{\underline{\mathrm{H}}}}}}
\newcommand{\Hpp}{{\ensuremath{{}^{\underline{\underline{\mathrm{H}}}}}}}
\newcommand{\ah}{{\ensuremath{A\Hp}}\relax}
\newcommand{\ahr}{{\ensuremath{A\Hpp}}\relax}
\newcommand{\acoh}{{\ensuremath{A^{\underline{\mathrm{coH}}}}}\relax}
\newcommand{\acohr}{{\ensuremath{A^{\underline{\underline{\mathrm{coH}}}}}}\relax}
\newcommand{\acohs}{{\ensuremath{A^{\underline{\mathrm{coH^\ast}}}}}\relax}
\newcommand{\acohsr}{{\ensuremath{A^{\underline{\underline{\mathrm{coH^\ast}}}}}}\relax}
\newcommand{\ahop}{{\ensuremath{\ah^{\mathrm{op}}}}}
\newcommand{\ashp}{{\ensuremath{A \Sp H}\relax}}
\newcommand{\ashpp}{{(\ashp)\Hp}}
\newcommand{\Ape}{\mathrel{\cdot}}
\newcommand{\Acao}{\mathrel{\triangleright}}
\newcommand{\M}{{\ensuremath{\mathcal{M}}}}
\newcommand{\Can}{\textrm{\textit{can}} }

\newcommand{\Not}[2][0]{%
	\setcounter{enumi}{#1}%
	\renewcommand{\theenumi}{#2\arabic{enumi}}%
	\renewcommand{\labelenumi}{(\theenumi)}%
	\settowidth{\itemindent}{\hskip-\labelsep#2}%
}

\let\originaleqnarray\eqnarray  
\def\eqnarray{\settowidth{\arraycolsep}{$\mskip 0.5\thickmuskip$}\originaleqnarray}
\expandafter\let\expandafter\eqnarraystar
\csname eqnarray*\endcsname  
\expandafter\def\csname eqnarray*\endcsname
{\settowidth{\arraycolsep}{$\mskip 0.5\thickmuskip$}\eqnarraystar}

\begin{document}\thispagestyle{empty}

\title{Partial Hopf-Galois theory}
\author[Castro, Freitas, Paques, Quadros, Tamusiunas]{Felipe Castro, Daiane Freitas, Antonio Paques, Glauber Quadros, Thaísa Tamusiunas}


\begin{abstract}
We develop a partial Hopf-Galois theory for partial $H$-module algebras and we recover analogs of classical results for Hopf algebras.
\end{abstract}

\maketitle

\



\section{Introduction}
The concept of partial actions and coactions of Hopf algebras on algebras were introduced by Caenepeel and Janssen in \cite{CJ}, generalizing left $H$-module algebras and right $H$-comodule algebras, respectively. 

A left partial action of the Hopf algebra $H$ on the algebra $A$ is a linear mapping $\alpha: H \otimes A \rightarrow A$, denoted here by $\alpha(h \otimes a) = h \cdot a$, such that
\begin{enumerate}\Not{PA}
    \item $h \cdot (ab) = (h_{1} \cdot a)(h_2 \cdot b)$,
    \item $ 1_H \cdot a = a$,
    \item $h \cdot (k \cdot a) = (h_1 \cdot 1_A)((h_2k) \cdot a), \forall h, k \in H, a \in A$.
\end{enumerate}
In this case, we call A a \emph{left partial H-module algebra}.

There is a class of partial actions in which the subalgebra of left invariants is equal to the subalgebra of right invariants, called \emph{symmetric partial actions}. We say that a partial action of $H$ on $A$ is \emph{symmetric} if it has the additional property:
\begin{enumerate}\Not[3]{PA}
    \item $h \cdot (k \cdot a) = ((h_1k) \cdot a)(h_2 \cdot 1_A), \forall h, k \in H, a \in A$.
\end{enumerate}
In this case, we call A a \emph{left symmetric partial H-module algebra}.

Our purpose in this paper is to construct equivalences to the definition of Hopf-Galois extension for this class of partial actions. Global actions are symmetric, so, in particular, we recover the equivalences developed in \cite{M}*{Theorem~8.3.3}. Also, given a globalizable partial action $\alpha$ of a group $G$ on the algebra $A$, the partial (Hopf) action of the group algebra $\Bbbk G$ induced by $G$ is symmetric \cite{CPQS}*{Lemma~3.4}. Thus symmetric partial actions of Hopf algebras extend (globalizable) partial actions of groups.

The paper is organized as follows. In \Cref{sec:preliminares} we present some preliminary results about the subalgebras of invariants and integrals. In \Cref{sec:Frobenius} we prove that if $A$ is a symmetric partial $H$-module algebra, then $A \Sp H / A$ is Frobenius and we use it to show a characterization about the fixed part of the partial smash product in terms of integrals. In fact, we show that $\ashpp = \left( 1 \Sp \int_{l}^H \right)(A \Sp 1_H)$, where $\int_{l}^H$ is the set of left integrals in $H$. \Cref{sec:Galois} is destined for Galois theory. In \Cref{galois} we give seven equivalences for the definition of Hopf-Galois partial extension.

For convention, unless otherwise specified, $H$ is a  Hopf algebra over a field $\Bbbk$ with antipode $S$ and $A$ is a left partial $H$-module algebra. Unadorned $\otimes$ means $\otimes_{\Bbbk}$. Rings and algebras are associative and unital. For all the paper we will adopt the Sweedler notation (summation understood).

\section{Preliminary}\label{sec:preliminares}

Here we present the background about invariants, coinvariants and integrals.

\subsection{Subalgebra of invariants and of coinvariants}

We begin this subsection setting the left and the right invariant parts of a partial $H$-module algebra $A$. Then we show that they coincide if $A$ is symmetric.

Define the following sets:
\begin{align*}
    \ah &:= \{ a\in A\mid h\Ape a = a (h\Ape 1_A), \forall h \in H \};\\
    \intertext{and}
    \ahr &:= \{ a\in A\mid h\Ape a = (h\Ape 1_A) a, \forall h \in H \}.
\end{align*}

We call $\ah$ the set of \emph{left invariants} and $\ahr$ the set of \emph{right invariants by the partial action of $H$ on $A$}.

\begin{lema}
    In the same notations above, $\ah\subseteq\ahr$ are subalgebras of $A$. Moreover, if the set $H\Ape 1_A = \{h \Ape 1_A \mid h \in H\}$ commutes with $\ahr$, then $\ah=\ahr$.
\end{lema}
\begin{proof}
    Let's check that $\ahr$ is a subalgebra of $A$. The other case follows analougosly. 

	Let $a, b\in\ahr$. Thus
	\begin{align*}
		h \Ape ab
			&= (h_1\Ape a)(h_2\Ape b)\\
			&= (h_1\Ape a)(h_2\Ape 1_A)b\\
			&= (h\Ape a) b\\
			&= (h\Ape 1_A) ab.
	\end{align*}
	Therefore, $ab\in\ahr$  and $\ahr$ is a subalgebra of $A$.
	
	To show that $\ah\subseteq\ahr$, take $a\in\ah$. Then
	\begin{align*}
		(h\Ape 1_A) a
			&= h_1\Ape[1_A (S(h_2)\Ape a)]\\
			&= h_1\Ape[a (S(h_2)\Ape 1_A)]\\
			&= (h_1\Ape a) (h_2 S(h_3)\Ape 1_A)\\
			&= h\Ape a.
	\end{align*}
	Thus, $\ah\subseteq\ahr$.
	
	Clearly, if $H\Ape 1_A$ commutes with $\ahr$, the equality $\ah=\ahr$ holds.
\end{proof}

\begin{prop}
	If the antipode $S$ is bijective and $A$ is a symmetric partial $H$-module algebra, then $H \Ape 1_A$ commutes with $\ahr$.
\end{prop}
\begin{proof}
Since the action is symmetric and the antipode is invertible, we have that 
		\[
			a(h\Ape b) = h_2\Ape[(S^{-1}(h_1)\Ape a) b]
		\]
	for all $a,b$ in $A$ and $h$ in $H$. Thus, taking $b=1_A$ and $a\in\ahr$, it follows that
	\begin{align*}
		a( h\Ape 1_A) 
			&= h_2 \Ape [(S^{-1}(h_1) \Ape a)1_A]\\
			&= h_2 \Ape [(S^{-1}(h_1) \Ape 1_A) a]\\
			&= 1_A (h\Ape a)\\
			&= (h\Ape 1_A)a
	\end{align*}
	and so $H \Ape 1_A$ commutes with $\ahr$.
\end{proof}

Our main interest is to work with symmetric partial actions, so we have in particular that $\ah=\ahr$. We may also refer to it as the \emph{fixed part} by the partial action of $H$ on $A$.


\let\e=\varepsilon
\def\0{{}^{\bar0}{}}
\def\1{{}^{\bar1}{}}

Now we recall from \cite{CJ} the definition of partial $H$-comodule algebra. Let $A$ be a $\Bbbk$-algebra and $\bar{\rho}: A \rightarrow A \otimes H$, $\bar{\rho}(a) = a\0 \otimes a\1$, for $a \in A$, a $\Bbbk$-linear map. We call $(A, \bar{\rho})$ a \emph{right partial $H$-comodule algebra} if, for any $a, b \in A$, 
\begin{enumerate}\Not{PCA}
    \item $(ab)\0 \otimes (ab)\1 = a\0b\0 \otimes a\1b\1$, \label{PCA2}
    \item $a\0\0 \otimes a\0\1 \otimes a\1 = a\0 1\0 \otimes a\1{}_1 1\1 \otimes a\1{}_{2}$, \label{PCA3}
    \item $\epsilon (a\1)a\0 = a.$ \label{PCA1}
\end{enumerate}

We say that the right partial $H$-comodule algebra $(A, \bar{\rho})$ is \emph{symmetric} if it satisfies the following additional condition:
\begin{enumerate}\Not[3]{PCA}
    \item $a\0\0 \otimes a\0\1 \otimes a\1 = 1\0 a\0 \otimes 1\1 a\1_1 \otimes a\1_2$.\label{PCA4}
\end{enumerate}

If $(A, \bar{\rho})$ is a right partial $H$-comodule algebra $H$, then we can define the set of coinvariants by the partial coaction. Consider:
\begin{align*}
    \acoh &= \{a \in A \mid \bar{\rho}(a) = a\bar{\rho}(1_A)\} = \left\{ a \in A \;\middle|\; a\0 \otimes a\1 = a 1\0 \otimes 1\1\right\};\\
    \intertext{and}
    \acohr &= \{a \in A \mid \bar{\rho}(a) = \bar{\rho}(1_A)a\} = \left\{ a \in A \;\middle|\; a\0 \otimes a\1 = 1\0 a \otimes 1\1 \right\}.
\end{align*}

We call $\acoh$ the set of \emph{left coinvariants} and $\acohr$ the set of \emph{right coinvariants by the partial coaction of $H$ on $A$}. Is is not difficult to check that $\acoh$ and $\acohr$ are both subalgebras of $A$.

The next result shows that, as well as in the case of partial actions, the symmetry of the partial coaction and the bijectivity of the antipode are sufficient conditions to have the equality between the left coinvariants and the right coinvariants.

\begin{prop}
    If the antipode $S$ is bijective and \((A, \bar\rho)\) is a symmetric right partial \(H\)-comodule algebra,  then \( \acoh = \acohr\).
\end{prop}
\begin{proof}
    Let \(a \in \acoh\). Then
    \[
        \begin{aligned}
        \bar{\rho}(1_A)a = 1\0 a \otimes 1\1
            &= 1\0 a\0 \otimes 1\1 \e(a\1)\\
            &= 1\0 a\0 \otimes 1\1 a\1_1 S(a\1_2)\\\text{symmetric p.c.} \to
            &= a\0\0 \otimes a\0\1 S(a\1)\\
            &= (a\0)\0 \otimes (a\0)\1 S(a\1)\\a \in \acoh \to
            &= (a 1\0)\0 \otimes (a 1\0)\1 S(1\1)\\
            &= (a 1' 1\0)\0 \otimes (a 1' 1\0)\1 S(1\1)\\
            &= a\0 1'\0 1\0\0 \otimes a\1 1'\1 1\0\1 S(1\1)\\\text{symmetric p.c.} \to
            &= a\0 1\0 \otimes a\1 1\1_1 S(1\1_2)\\
            &= a\0 1\0 \otimes a\1 \e(1\1)\\
            &= a\0 \otimes a\1 = \bar{\rho}(a)
        \end{aligned}
    \]
    
    Therefore, \(a \in \acohr\). By the other hand, let \(b \in \acohr\), then
    \[
        \begin{aligned}
       b\bar{\rho}(1_A) = b 1\0 \otimes 1\1
            &= b\0 1\0 \otimes \e(b\1) 1\1\\
            &= b\0 1\0 \otimes S^{-1}(b\1_2) b\1_1 1\1\\
            &= b\0\0 \otimes S^{-1}(b\1) b\0\1\\
            &= (b\0)\0 \otimes S^{-1}(b\1) (b\0)\1\\b \in \acohr \to
            &= (1\0 b)\0 \otimes S^{-1}(1\1) (1\0 b)\1\\
            &= (1\0 1' b)\0 \otimes S^{-1}(1\1) (1\0 1' b)\1\\
            &= 1\0\0 1'\0 b\0 \otimes S^{-1}(1\1) 1\0\1 1'\1 b\1\\
            &= 1\0 b\0 \otimes S^{-1}(1\1_2) 1\1_1 b\1\\
            &= 1\0 b\0 \otimes \e(1\1) b\1\\
            &= b\0 \otimes b\1 = \bar{\rho}(b).
        \end{aligned}
    \]
    Therefore, \(b \in \acoh\) and, consequently, \(\acoh = \acohr\).
\end{proof}

\subsection{Integrals}

In this subsection we have the fundamental results about integrals for a good understanding about the work.

An element $t$ in $H$ is called a \emph{left integral} if for any $h \in H$ we have that $ht = \varepsilon(h)t$. A \emph{right integral} in $H$ is defined in a completely analogous way.

\vspace{0,1cm}

\noindent \textbf{Notation:} We will denote by $\int_{l}^H$ the set of left integrals and by $\int_{r}^H$ the set of right integrals in $H$.

\vspace{0,1cm}

We shall relate the integrals in the dual algebra $H^\ast$ of $H$ with the integrals in $H$. So, for the rest of this subsection, suppose that $H$ is finite-dimensional (which implies that $H^\ast$ has a Hopf algebra structure and that the antipode $S$ is bijective). Observe that an integral $T\in \int_{l}^{H^*}$ is such that:
\begin{eqnarray*}
	f \ast T
		&=& \varepsilon_{H^\ast}(f)T\\
		&=& f(1_H)T,
\end{eqnarray*} for all $f \in H^*$, where $\ast$ is the convolution product.

\begin{prop}\label{th1h2}
	Let $T$ be an element in $H^\ast$. Thus:
	\begin{enumerate}[(i)]
		\item if $T \in \int_{r}^{H^*}$, then $T(h_1)h_2 = T(h)1_H$ for all $h \in H$;
		\item if $T \in \int_{l}^{H^*}$, then $h_1T(h_2) = T(h)1_H$ for all $h \in H$.
	\end{enumerate}
\end{prop}

\begin{proof}
	For (i), let $T \in \int_{r}^{H^\ast}$. So for all $h \in H$ and $f \in H^\ast$ we have that 
	\begin{eqnarray*}
	f(T(h_1)h_2)
		&=& T(h_1)f(h_2)\\
		&=& T \ast f (h)\\
		&=& \varepsilon_{H^\ast}(f) T(h)\\
		&=& f(1_H) T(h)\\
		&=& f(T(h)1_H).
	\end{eqnarray*}
	Therefore $T(h_1)h_2 = T(h)1_H$ for all $h \in H$.
	
	The item (ii) is analogous.
\end{proof}

\begin{prop}
	Let $T \in H^\ast$ and $t \in H$ with $T(t) = 1$. Thus:
	\begin{enumerate}[(i)]
		\item if $T \in \int_{r}^{H^\ast}$ and $t \in \int_l^{H}$, then $T(S(h)t_1)t_2 = h$ for all $h \in H$;
		\item if $T \in \int_{l}^{H^\ast}$ and $t \in \int_r^{H}$, then $t_1 T(t_2S(h)) = h$ for all $h \in H$.
	\end{enumerate}
\end{prop}

\begin{proof}
	We prove the first statement. The other one is completely analogous. Let $T \in \int_{l}^{H^\ast}$. Therefore, for all $h \in H$, we have that
	\begin{eqnarray*}
	S(T(S(h)t_1)t_2) &=& T(S(h)t_1)S(t_2) \\
	&\overset{\ref{th1h2}}{=}& T(S(h)_1 t_1)S(h)_2 t_2 S(t_3) \\
	&=& T(S(h)_1 t)S(h)_2 \\
	&=& T(\varepsilon(S(h)_1) t)S(h)_2 \\
	&=&T(t) S(h) \\
	&=& S(h).
	\end{eqnarray*}

	Since $S$ is bijective, the result follows.
\end{proof}

\begin{obs}\label{tt1}
	Notice that if $t \in \int_l^{H}$, we have the following isomorphism:
	\begin{eqnarray*}
	H^\ast & \longrightarrow & H\\
	f & \longmapsto & f(t_1) t_2.
	\end{eqnarray*}
	Then, if $T\in \int_r^{H^\ast}$ is nonzero, it follows that $T(t)1_H = T(t_1)t_2 \neq 0$, and consequently $T(t)$ is distinct of zero.
	
	The same occurs if $T \in \int_l^{H^\ast}$, using the isomorphism:
	\begin{eqnarray*}
	H^\ast & \longrightarrow & H\\
	f & \longmapsto & t_1 f(t_2).
	\end{eqnarray*}
	
	Therefore, it is always possible to choose integrals $t$ and $T$ such that $T(t) = 1$. 
\end{obs}

\section{Frobenius for partial actions of Hopf algebras}\label{sec:Frobenius}

For the rest of the paper, unless otherwise specified, we fix that the Hopf algebra $H$ is finite-dimensional. Notice that in this case the antipode of $H$ is bijective.

\subsection{Frobenius and partial actions} 
Recall that the smash product $A \# H$ is the tensor product $A \otimes H$ with multiplication rule given by $(a\#h)(b\#g) = a(h_1 \cdot b)\#h_2g$. According to  \cite{CJ}, we set the \emph{partial smash product} $A \Sp H$ as the subalgebra generated by the elements of the form $(a\#h)(1_A \# 1_H) = a(h_1 \cdot 1_A)\#h_2$.

The purpose of this subsection is to prove that if $A$ is a symmetric partial $H$-module algebra, then $A \Sp H / A$ is Frobenius. We start with the definition of a Frobenius ring extension $S / R$.

\begin{dfn}\label{frobenius}
	Given $\imath\colon R \to S$ a ring homomorphism, we say that $\imath$ is \emph{Frobenius} (or that the ring extension $S / R$ is Frobenius) if there exists a pair $(\Phi,e)$, called \emph{Frobenius system}, such that $\Phi\colon S \to R$ is an $R$-bimodule homomorphism and $e = \sum\limits_{i} x_i \otimes y_i \in S \otimes_R S$ satisfies $es = se$, for all $s \in S$, and, $\sum\limits_{i}\Phi(x_i)y_i = \sum\limits_{i} x_i\Phi(y_i) = 1_S$.
\end{dfn}


\begin{ex}
	Let $t \in \int_l^{H}$ and $T\in \int_r^{H^\ast}$ be such that $T(t) = 1$. Then $H / \Bbbk$ is Frobenius with $\Phi = T$ and $e = t_2 \otimes S^{-1}(t_1)$.
\end{ex}

Assuming that $A$ is a partial $H$-module algebra, $h \Ape 1_A$ is central in $A$ and
\[
	t_1 \otimes t_2 \otimes t_3 \otimes t_4  = t_1 \otimes t_3 \otimes t_2 \otimes t_4,
\] S. Caenepeel e K. Janssen showed in \cite{CJ} that $A \Sp H / A$ is Frobenius.  

Here we deal with a different case. Supposing only that $A$ is a symmetric partial $H$-module algebra (that is, without restrictions on the co-product of the integral $t$, neither on the centrality of $h \Ape 1_A$) we will prove that $A \Sp H / A$ is Frobenius. So from now on we assume that $A$ is a left symmetric partial $H$-module algebra.

Before we go any further, we need to make a brief observation about the definition of integrals, that will be useful in the construction of this section.

\begin{obs}
	By definition, $t$ is a left integral in $H$, if for any $h \in H$ we have that $h t = \varepsilon(h) t$. However, equivalently to this definition, we have that $t$ is a left integral in $H$ if for any $h \in H$ we have that
	\begin{eqnarray}\label{defequivint}
	t_1 \otimes h t_2 = S(h) t_1 \otimes t_2
	\end{eqnarray}
	 and 
	\begin{eqnarray}\label{defequivintHdimfinita}
	S^{-1}(t_1) \otimes h t_2 = S^{-1}(t_1) h \otimes t_2.
	\end{eqnarray}
	
\end{obs}

We observed in Remark \ref{tt1}, there exist a left integral $t$ in $H$ and a right integral $T$ in $H^\ast$ such that $T(t) = 1$. Thus from now on the integrals $(t,T)$ have this property.

Our main goal of this subsection is to show that $A \Sp H / A$ is Frobenius. We will proceed constructively. We first construct the element $e$ in $(A \Sp H) \otimes_A (A \Sp H)$  and then we define the map $\Phi\colon A\Sp H \to A$.

Remember that the map
\begin{eqnarray*}
\imath:A &\to& A \Sp H\\
a & \mapsto& a\Sp 1
\end{eqnarray*}
is a monomorphism of unital algebras. Therefore we can identify the elements of $A$ as elements of $A \Sp H$. Then, from now on, to avoid excessive formalism, we treat an element of $A$ as the form $a \Sp 1$. Thus, $A \Sp H$ becomes an $A$-bimodule via multiplication.

For what follows, we set the following element $$e = (1 \Sp t_2)\otimes_A (1 \Sp S^{-1}(t_1)) \in (A \Sp H) \otimes_A (A \Sp H).$$

\begin{prop}\label{e1e2r=re1e2}
	For any element $a \Sp h$ in $A \Sp H$ we have that $e(a \Sp h) = (a \Sp h)e$, that is:
	\[
		(1 \Sp t_2)\otimes_A [(1 \Sp S^{-1}(t_1))(a \Sp h)] = [(a \Sp h)(1 \Sp t_2)] \otimes_A (1 \Sp S^{-1}(t_1)).
	\]
\end{prop}
\begin{proof}
	Take $a \Sp h$ in $A \Sp H$. Then:
	\begin{eqnarray*}
		1\Sp t_2 \otimes_A [(1 \Sp S^{-1}(t_1))(a \Sp h)]
		    &=& 1\Sp t_2 \otimes_A (1 \Sp S^{-1}(t_1))(a (h_1 \Ape 1) \Sp h_2)\\
			&=& 1\Sp t_3 \otimes_A S^{-1}(t_2) \Ape a (h_1 \Ape 1) \Sp S^{-1}(t_1)h_2\\
			&=& 1\Sp t_3 \otimes_A [S^{-1}(t_2) \Ape a (h_1 \Ape 1) \Sp 1][1 \Sp S^{-1}(t_1)h_2]\\
			&=& (1\Sp t_3)[S^{-1}(t_2) \Ape a (h_1 \Ape 1) \Sp 1] \otimes_A 1 \Sp S^{-1}(t_1)h_2\\
			&=& t_3 \Ape S^{-1}(t_2) \Ape a (h_1 \Ape 1) \Sp t_4 \otimes_A 1 \Sp S^{-1}(t_1)h_2\\
		{\rm symmetric~ p.a.}\rightarrow
			&=& [t_3 S^{-1}(t_2) \Ape a (h_1 \Ape 1)](t_4 \Ape 1) \Sp t_5 \otimes_A 1 \Sp S^{-1}(t_1)h_2\\
			&=& [\varepsilon(t_2) a (h_1 \Ape 1)](t_3 \Ape 1) \Sp t_5 \otimes_A 1 \Sp S^{-1}(t_1)h_2\\
			&=& a (h_1 \Ape 1)(t_2 \Ape 1) \Sp t_3 \otimes_A 1 \Sp S^{-1}(t_1)h_2\\
			& = & [(a (h_1 \Ape 1)\Sp t_2)(1 \Sp 1) ]\otimes_A 1 \Sp S^{-1}(t_1)h_2\\
			& = & a (h_1 \Ape 1)\Sp t_2 \otimes_A 1 \Sp S^{-1}(t_1)h_2\\
			&\overset{(\ref{defequivintHdimfinita})}{=} & a (h_1 \Ape 1) \Sp h_2 ~ t_2 \otimes_A 1 \Sp S^{-1}(t_1)\\
			& = & (a \Sp h)(1 \Sp t_2) \otimes_A 1 \Sp S^{-1}(t_1)
	\end{eqnarray*}
	as we wanted.
\end{proof}

Now that we have the element $e$, we need to construct the map $\Phi$ such that we have a Frobenius system.

Define
\begin{eqnarray*}
\Phi\colon A \Sp H & \longrightarrow & A\\
a \Sp h & \longmapsto & a T(h)
\end{eqnarray*}
as a candidate to our map.

\begin{prop}\label{PhiAbil}
	$\Phi$ is an $A$-bimodule homomorphism.
\end{prop}

\begin{proof}
	Let $b \in A$ and $a \Sp h \in A \Sp H$. Notice that we can see $b$ as $b \Sp 1$, which give us the $A$-bimodule structure of $A \Sp H$. Therefore,
	\begin{eqnarray*}
	\Phi((b \Sp 1)(a \Sp h)) & = & \Phi(ba \Sp h) \\
	& = & ba T(h)\\
	& = & b \Phi(a \Sp h)
	\end{eqnarray*} 
	which means that $\Phi$ is left $A$-linear.
	
	Furthermore,
	\begin{eqnarray*}
	\Phi((a \Sp h)(b \Sp 1)) & = & \Phi(a(h_1 \Ape b) \Sp h_2) \\
	& = & a(h_1 \Ape b) T(h_2)\\
	& = & a(h_1T(h_2) \Ape b)\\
	& = & a(T(h)1_H \Ape b)\\
	& = & a(1_H \Ape b) T(h)\\
	& = & abT(h)\\
	& = & aT(h) b\\
	& = & \Phi(a \Sp h) b
	\end{eqnarray*} 
	and then $\Phi$ is right $A$-linear.
\end{proof}

It only remains to show the relation between $\Phi$ and $e$ described in Definition \ref{frobenius}, that is, that the following proposition is valid.

\begin{prop}\label{Phiee}
	Following the same notations, we have:
	\begin{enumerate}[(i)]
		\item $\Phi(1_A \Sp t_2)  (1_A \Sp S^{-1}(t_1)) = 1_A \Sp 1_H$;\label{Phiee1}\\
		\item $(1_A \Sp t_2) \Phi(1_A \Sp S^{-1}(t_1)) = 1_A \Sp 1_H$.\label{Phiee2}
	\end{enumerate}
\end{prop}

\begin{proof}
	For \eqref{Phiee1}, we have that: 
	\begin{eqnarray*}
	\Phi(1_A \Sp t_2)  (1_A \Sp S^{-1}(t_1))
	& = & T(t_2) (1_A \Sp S^{-1}(t_1))\\
	& = & 1_A \Sp S^{-1}(t_1 T(t_2))\\
	& = & 1_A \Sp S^{-1}(T(t)1_H)\\
	& = & 1_A \Sp S^{-1}(1_H)\\
	& = & 1_A \Sp 1_H.
	\end{eqnarray*}

	For \eqref{Phiee2}, remember that $H / \Bbbk$ is Frobenius with $\Phi = T$ and $e = t_2 \otimes S^{-1}(t_1)$. In particular, it follows that $t_2 T(S^{-1}(t_1)) = 1_H$. Then:
	\begin{eqnarray*}
	(1_A \Sp t_2) \Phi(1_A \Sp S^{-1}(t_1))
	& = & (1_A \Sp t_2) T(S^{-1}(t_1))\\
	& = & 1_A \Sp t_2T(S^{-1}(t_1))\\
	& = & 1_A \Sp 1_H.
	\end{eqnarray*}
\end{proof}

\begin{obs}
	Note that the statements in the above proposition can be rewritten as: 
	\begin{description}
		\item[(\ref{Phiee1})] $(\Phi \otimes_A I)(e) = 1_A \otimes_A 1_A \Sp 1_H$;\\
		\item[(\ref{Phiee2})] $(I \otimes_A \Phi)(e) = 1_A \Sp 1_H \otimes_A 1_A$.
	\end{description}
\end{obs}

So we have the following:

\begin{teo}
	Let $A$ be left symmetric partial H-module algebra. Then $A\Sp H / A$ is Frobenius with Frobenius system $(\Phi,e)$ given by:
	\begin{eqnarray*}
	\Phi\colon A \Sp H & \longrightarrow & A\\
	a \Sp h & \longmapsto & a T(h)
	\end{eqnarray*}
	and
	\begin{equation*}
	e = (1 \Sp t_2)\otimes_A (1 \Sp S^{-1}(t_1)).
	\end{equation*}
\end{teo}

\subsection{Application of the Frobenius Theory}

In the case of global actions of Hopf algebras, given $A$ a left $H$-module algebra (which is always symmetric), we have that $A \# H$ is a left $H$-module via the immersion of $H$ in $A \# H$ as algebra, that is, we have an action $\triangleright\colon H \otimes (A \# H) \to A\# H$ given by $g \triangleright (a \# h) = (1_A \# g)(a \# h)$.

Furthermore, the fixed part by this action is given by
\begin{eqnarray*}
(A \# H)^H & = & \{a \# h \in A \# H | g \triangleright (a \# h) = (g \Ape 1_A)a \# h\}\\
& = & \{a \# h \in A \# H | (1_A \# g)(a \# h) = (g \Ape 1_A)a \# h\}.
\end{eqnarray*}

By \cite{L}*{Lemma~4.3} we can characterize this fixed part in terms of left integrals as we describe below. 

\begin{equation}
\left(1_A \# \int_{l}^H\right)(A \# 1_H) = (A \# H)^H.
\end{equation} 

This equality is still valid, \emph{mutatis mutantis}, in the partial case. However, a turning around Frobenius is necessary. 

Notice that a first problem is that $H$ not necessarily immerses in $A \Sp H$ as algebra, implying that $A \Sp H$ do not become a $H$-module. Nevertheless, we can take the $\Bbbk$-vector space:
\begin{eqnarray*}
\ashpp & = & \{a \# h \in A \# H \mid (1_A \# g)(a \# h) = (g \Ape 1_A)a \# h\}.
\end{eqnarray*}

We shall show that:
\begin{equation}
\left( 1_A \Sp \int_{l}^H \right)(A \Sp 1_H) = \ashpp.
\end{equation}

Recall that we fixed that the left integral $t$ in $H$ and the right integral $T$ in $H^\ast$ are such that $T(t) = 1$.

Consider, then, $(\Phi,e)$ the Frobenius system for $A \Sp H / A$ previously described.

We define maps $\alpha$ and $\beta$ as
\begin{eqnarray*}
\alpha\colon A & \longrightarrow & \ashpp \\
a & \longmapsto & (1_A \Sp t)(a \Sp 1_H)
\end{eqnarray*}
and $\beta$ as the restriction of $\Phi$ to $\ashpp$, that is,
\begin{eqnarray*}
\beta\colon \ashpp & \longrightarrow & A \\
a \Sp h & \longmapsto & \Phi(a \Sp h) = a T(h).
\end{eqnarray*}

\begin{obs}
	Since $\dim_{\Bbbk} \int_l^H = 1$, for any $a \in A$, $h \in H$ and $s \in \int_l^H$, we have that:
	\[
		(1_A \Sp h)(1_A \Sp s)(a \Sp 1_H) = (h_1 \Ape 1_A \Sp h_2 s) (a \Sp 1_H) = (h \Ape 1_A \Sp 1_H)(1_A \Sp s)(a \Sp 1_H).
	\]
	In other words, 
	\begin{eqnarray}\label{contidonofixo}
	\left( 1_A \Sp \int_l^H \right)(A \Sp 1_H) \subseteq \ashpp
	\end{eqnarray}
	which guarantees that $\alpha$ is well defined.
\end{obs}

\begin{prop}
	$A$ is isomorphic to $\ashpp$\ (as a left \ah-module) via $\alpha$, with inverse $\beta$.
\end{prop}

\begin{proof}
	First note that \ashpp\ is a left \ah-module via the multiplication of $A$. In fact, given $x\in\ah$ and $a\Sp h\in\ashpp$, we have that, for all $k\in H$,
		\begin{eqnarray*}
			(1_A\Sp k)\cdot(x \cdot (a\Sp h))
				&=& (1_A\Sp k) (xa\Sp h)\\
				&=& ((k_1\Ape xa)\Sp k_2h)\\
			{x\in\ah \rightarrow}
				&=& (x(k_1\Ape a)\Sp k_2h)\\
				&=& (x\Sp 1_H)((k_1\Ape a)\Sp k_2h)\\
				&=& (x\Sp 1_H)(1_A\Sp k) (a\Sp h)\\
			{a\Sp h\in\ashpp \rightarrow}
				&=& (x\Sp 1_H)((k\Ape 1_A)\Sp 1_H) (a\Sp h)\\
				&=& (x(k\Ape 1_A)\Sp 1_H) (a\Sp h)\\
			{x\in\ah \rightarrow}
				&=& ((k\Ape 1_A)x\Sp 1_H) (a\Sp h)\\
				&=& ((k\Ape 1_A)\Sp 1_H)(x\Sp 1_H)(a\Sp h)\\
				&=& ((k\Ape 1_A)\Sp 1_H)\cdot(x\cdot(a\Sp h)).
		\end{eqnarray*}
	It follows by the definition of $\alpha$ that it is left \ah-linear.
	
	Now we will see that $\alpha$ and $\beta$ are mutually inverse. In fact, notice that for any $a \in A$ we have that
	\begin{eqnarray*}
	\beta(\alpha(a)) &=& \Phi(\alpha(a))\\
	&=& \Phi((1_A \Sp t)(a \Sp 1_H))\\
	&\overset{\ref{PhiAbil}}{=}& \Phi(1_A \Sp t)a\\
	&=& 1_A T(t) a\\
	&=& a.
	\end{eqnarray*}
	Moreover, if $a \Sp h \in \ashpp$,
	\begin{eqnarray*}
		\alpha(\beta(a \Sp h))
			&=& (1_A \Sp t)(\beta(a \Sp h)\Sp 1_H)\\
			&=& (1_A \Sp t)(\Phi(a \Sp h)\Sp 1_H)\\
			&=& (1_A \Sp t_2)((S^{-1}(t_1) \Ape 1_A) \Sp 1_H)(\Phi(a \Sp h) \Sp 1_H) \\
			&\overset{\ref{PhiAbil}}{=}& (1_A \Sp t_2)(\Phi(((S^{-1}(t_1) \Ape 1_A)\Sp 1_H)(a \Sp h)) \Sp 1_H)\\
		{a \Sp h \in \ashpp \rightarrow}
			&=& (1_A \Sp t_2)[\Phi((1 \Sp S^{-1}(t_1)) (a \Sp h)) \Sp 1_H]\\
			&\overset{\ref{e1e2r=re1e2}}{=}& (a \Sp h) (1_A \Sp t_2) [\Phi(1 \Sp S^{-1}(t_1)) \Sp 1_H]\\
			&\overset{\ref{Phiee} - (\ref{Phiee2})}{=}& a \Sp h.
	\end{eqnarray*}
\end{proof}

Therefore, given $a \Sp h \in \ashpp$, there exists $b \in A$ such that $\alpha(b) = a \Sp h$, that is, $(1_A \Sp t)(b \Sp 1_H) = a \Sp h$. Thus $\ashpp \subseteq \left( 1_A \Sp \int_{l}^H \right)(A \Sp 1_H)$ and, additionally with \eqref{contidonofixo}, it follows that
\begin{eqnarray}\label{aigualdade}
    \ashpp = \left( 1 \Sp \int_l^H \right)(A \Sp 1_H).
\end{eqnarray}

\section{Galois Theory}\label{sec:Galois}

The goal of this section is to prove several equivalences of the definition of Galois extension for a finite dimensional Hopf algebra acting symmetric partially on an algebra. 


Let $t$ be a nonzero left integral in $H$. Recall that we assume in the previous section that $H$ is a finite-dimensional Hopf algebra over $\Bbbk$. Then we have the following isomorphism (of $H^\ast$-modules):
\[
	\begin{alignedat}1
		\theta\colon H^\ast&\longrightarrow H\\
		f&\longmapsto t \leftharpoonup f = f(t_1) t_2.
	\end{alignedat}
\]

Considering $A$ a left partial $H$-module algebra, we have that $A$ is a right partial $H^\ast$-comodule algebra (cf. \cites{CJ,CQ}) with induced coaction given by
\begin{align*}
	\bar{\rho}\colon A	& \longrightarrow A \otimes H^\ast\\
	a 					& \longmapsto \sum_i (h_i \Ape a) \otimes h^\ast_i = a^{0} \otimes a^{1},
\end{align*}
where $\{h_i, h^\ast_i\}$ is a dual basis with relation to $H$ and $H^\ast$. Moreover, notice that this coaction satisfies the following compatibility rule:
\begin{equation}
	h \Ape a = a^0 a^1(h)\label{eq:1}
\end{equation}
In this case, we can consider the reduced tensor $A\Tr H^\ast$, defined in \cite{BV}, which is the following subspace of $A \otimes H^\ast$:
\[
	A \Tr H^\ast = (A \otimes H^\ast) \bar{\rho}(1_A) = \{a 1^0 \otimes f \ast 1^1\mid a\in A, f\in H^\ast\}.
\]


\begin{lema}\label{lema-1}
	Let $A$ be a partial $H$-module algebra. Then $A \Tr H^\ast$ and $A\Sp H$ are isomorphic as left $A$-modules.
\end{lema}

\begin{proof}
		Let $t\neq0$ a left integral in $H$. Thus, considering the isomorphism $\theta$ between $H^\ast$ and $H$, we have that the map
		\[
			\begin{alignedat}1
				I\otimes\theta\colon A\otimes H^\ast	& \longrightarrow A \# H\\
				a \otimes f								& \longmapsto a \mathrel{\#} \theta(f)
			\end{alignedat}
		\]
		is also a bijection.

		Now, noting that $(I\otimes\theta)(A\Tr H^\ast) = A\Sp H$, we can take
		\[
			\varphi\colon A\Tr H^\ast \longrightarrow A\Sp H
		\]
		as a restriction of $I\otimes\theta$ to $A\Tr H^\ast$, getting the desired isomorphism. Indeed, given $a\in A$ and $f\in H^\ast$, we have that
		\begin{eqnarray*}
			(I \otimes \theta)(a \Tr f)
				& = & (I \otimes \theta)(a 1^0 \otimes f \ast 1^1)\\
				& = & a 1^0 \# (f \ast 1^1) (t_1) t_2 \\
				& = & a 1^0 1^1(t_2) f(t_1) \# t_3 \\
				& \overset{\eqref{eq:1}}{=} & a (t_2 \Ape 1_A ) f(t_1)\# t_3 \\
				& = & (a \# f(t_1)t_2)(1_A \# 1_H)\\
				& = & a \Sp \theta(t).
		\end{eqnarray*}
		
		Since the structures of left $A$-modules are via multiplication in $A$, it follows that $A\Tr H^\ast \simeq A\Sp H$ as $A$-modules.
\end{proof}

Now we finally develop a Galois theory, in order to give several equivalences of the definition of Hopf-Galois partial extension.

\begin{dfn}Let $(A, \bar{\rho})$ be a right partial $H$-comodule algebra. The extension $\acoh\subseteq A$ is said to be  a \emph{Hopf-Galois partial extension} if the canonical map $\Can\colon A\otimes_{\acoh} A\to A\Tr H$ given by $\Can(a\otimes b) = a b^0 \Tr b^1$ is a bijection.

\end{dfn}

\begin{obs}
By \cite{AB}*{Lemma 3}, if $\bar{\rho}$ is the right partial $H^*$-comodule structure of $A$ induced by its left partial $H$-module structure, then $\ah = \acohs$. With an argument symmetrical to what was given in the proof of that result, we have that $\ahr = \acohsr$. Thus, if $A$ is a symmetric $H$-module algebra, it follows that $\acohs = \acohsr$. 
\end{obs}

Since $\ah = \acohs$, the canonical map can be considered as $\Can\colon A\otimes_{\ah} A\to A\Tr H^\ast$. Furthermore, \cite{AB}*{Theorem 2} ensures that if \Can is surjective, then it is a bijection.

Now we define the following linear map:
\[
	\begin{alignedat}1
		[~,~]\colon A\otimes_{\ah} A	&\longrightarrow A\Sp H\\
		a\otimes b 					& \longmapsto atb = (a\Sp 1_H)(1_A\Sp t)(b\Sp 1_H)
	\end{alignedat}
\]

By Lemma \ref{lema-1}, we have the first Galois equivalence, as follows.

\begin{prop}\label{cancolchete}
	The canonical map $\Can\colon A\otimes_{\ah} A\to A\Tr H^\ast$ is surjective if and only if $[~,~]\colon A\otimes_{\ah} A\to A\Sp H$ is surjective.
\end{prop}

\begin{proof}
	From the definition of $[~,~]$, we have the following commutative diagram:
	\[
	\xymatrix{
		&\ashp&\\
		A\otimes_{\ah} A \ar[ur]^{[ , ]} \ar[rr]^{\Can}&& A\Tr H^\ast \ar[ul]_\varphi\\
	}
	\]
	In fact, given $a, b\in A$, we have that
	\begin{eqnarray*}
		\varphi(\Can(a\otimes b))
		& = & \varphi (a b^0 \Tr b^1)\\
		& = & a b^0 \Sp \theta(b^1)\\
		& = & a b^0 \Sp b^1(t_1)t_2\\
		& = & a b^0 b^1(t_1) \Sp t_2\\
		& = & a (t_1\Ape b) \Sp t_2\\
		& = & (a \Sp 1_H) (1_A \Sp t)(b \Sp 1_H)\\
  		& = & [~,~](a \otimes b)
	\end{eqnarray*}
	Thus, $[~,~] = \varphi~\circ~ \Can$. Therefore $[~,~]$ is surjective if and only if \Can is surjective.
\end{proof}

We now start a construction in order to obtain other equivalences to the definition of Hopf-Galois partial extension. 

Let ${}_\ashp\mathcal{M}$ be the category of the left $\ashp$-modules.

\begin{lema}\label{lema-3}
	If $[~,~]$ is surjective, then $A$ is a generator in ${}_\ashp\mathcal{M}$.
\end{lema}

\begin{proof}
Assume that $[~,~]$ is surjective. Then there exist $\{b_i, c_i\}_{i=1}^n \subseteq A$ such that $1_A \Sp 1_H = \sum\limits_{i = 1}^n b_itc_i$. Define $\psi\colon A^{(n)} \to \ashp$ by $\psi(a_1, \cdots, a_n) = \sum\limits_{i = 1}^n a_itc_i$. Clearly, $\psi$ is additive. Observe that $A$ is a left $\ashp$-module via $(a\Sp h) \triangleright b = a(h \cdot b)$. Also notice that $A^{(n)}$ has a left $\ashp$-module structure induced naturally by the left $\ashp$-module structure of $A$. We will represent this action with the same symbol $\triangleright$. We shall prove that:

\noindent \textbf{(i)} $\psi$ is a left $\ashp$-linear map.

Indeed,

	\begin{eqnarray*}
			\psi((a\Sp h)\triangleright(a_1, \cdots, a_n))
				&=& \psi((a\Sp h)\triangleright a_1, \cdots, (a\Sp h)\triangleright a_n))\\
				&=& \psi(a(h \cdot a_1), \cdots, a(h \cdot a_n))\\
				&=& \sum\limits_{i = 1}^n a(h \cdot a_i)tc_i) = \sum_{i = 1}^n (a[(h \cdot a_i)t]c_i\\
				&=& \sum\limits_{i = 1}^n a[((h \cdot a_i)\Sp 1_H)(1_A \Sp t)]c_i\\
				&=& \sum\limits_{i = 1}^n a[(h \cdot a_i)\Sp t]c_i\\
				&=& \sum\limits_{i = 1}^n a[(h_1 \cdot a_i)\Sp \epsilon_H(h_2)t]c_i\\
				&=& \sum\limits_{i = 1}^n a[(h_1 \cdot a_i)\Sp h_2t]c_i\\
				&=& \sum\limits_{i = 1}^n a[(1_A \Sp h)(a_i \Sp t)]c_i\\
				&=& \sum\limits_{i = 1}^n (a \Sp h)(a_itc_i) = (a \Sp h)\sum\limits_{i = 1}^n (a_itc_i) \\
				&=& (a \Sp h)\psi(a_1, \cdots, a_n).
		\end{eqnarray*}

\noindent \textbf{(ii)} $\psi$ is surjective. 

In fact, let $a\Sp h \in \ashp$. Since $1_A \Sp 1_H = \sum\limits_{i = 1}^n b_itc_i$, we have that $$1_A \Sp 1_H = \psi(b_1, \cdots, b_n).$$Thus 

\begin{eqnarray*}
\psi((a\Sp h)\triangleright(b_1, \cdots, b_n)) 
				&=& (a\Sp h)\psi(b_1, \cdots, b_n)\\
				&=&  (a\Sp h)(1_A \Sp 1_H) = a\Sp h.
\end{eqnarray*}

Therefore $A$ is a generator in ${}_\ashp\mathcal{M}$.

\end{proof}
	
	


Before we proceed with the next result, we need to define an important $\Bbbk$-vector space. We already know that $H$ cannot be immersed in \ashp, so given a \ashp-module $M$, it does not have the ``induced" structure of $H$-module. Althought $M$ is not a $H$-module, we will use the same notation for the fixed part of $M$ by the action of $H$, given by the following subspace of $M$:
	\[
		M\Hp = \{ m\in M\mid (1_A\Sp h)\cdot m = ((h\Ape 1_A)\Sp 1_H)\cdot m,~\forall h\in H\}
	\]

\begin{obs}\label{obs:parte-fixa}
	In the global case, we have that the fixed part of $M$ by $H$ is given by
		\[
			M^{\rm H} = \{ m\in M\mid h\cdot m = \varepsilon(h) m,~\forall h\in H\}.
		\]
	However, since $H$ is immersed in $A\# H$ via $\imath\colon h\mapsto 1_A\# h$ and since $h \triangleright 1_A = \varepsilon(h) 1_A$, we have that
		\begin{align*}
			h \cdot m
				&= (1_A\# h) \cdot m\\
	\intertext{and}
			\varepsilon(h)m
				&= \varepsilon(h)(1_A \# 1_H) \cdot m\\
				&= ((h\triangleright 1_A)\# 1_H) \cdot m.
		\end{align*}
		
	Therefore $M^{\rm H}$ is exactly the same described in the partial case.
\end{obs}

\begin{lema}\label{lema-4}
	Suppose that $A$ is a finitely generated projective right $\ah$-module and $\Pi\colon \ashp \to End(A_{\ah})$ given by $\Pi(a \Sp h)(b) =  (a \Sp h) \Acao b$ is an isomorphism of algebras and of left $A$-modules. Then, for all $M \in _{\ashp} \M$, the map $\mu\ud{M}\!\colon A \otimes_{\ah} M\Hp \to M$ is an isomorphism of left $\ashp$-modules.
\end{lema}
\begin{proof}
	Since $A$ is a finitely generated projective right $\ah$-module, there exists a finite dual basis, that is, there exist $a_i\in A$ and $f_i\in{\rm Hom}_\ah(A, \ah)$, with $i\in\{1,\ldots,n\}$, such that, for all $a\in A$,
	\[
		a = \sum\limits_{i=1}^{n} a_i f_i(a)
	\]
	Since $\Pi$ is surjective, there exist $x_1,\ldots, x_n$ in \ashp\ such that $\pi(x_i) = f_i$, for all $i=1,\cdots,n$.
	
	Notice that these $x_i$'s are in \ashpp. Indeed, given $h\in H$ and $a\in A$, we have that
	\begin{eqnarray*}
		\Pi((1_A \Sp h) x_i)(a)
			&=& \Pi(1_A \Sp h)(\Pi(x_i)(a))\\
			&=& \Pi(1_A \Sp h)(f_i(a))\\
			&=& 1_A (h \Ape f_i(a))\\
			&=& (h \Ape 1_A)f_i(a)\\
			&=& (h \Ape 1_A)\Pi(x_i)(a)\\
			&=& \Pi(((h \Ape 1_A)\Sp 1_H)x_i)(a).
	\end{eqnarray*}Since $a$ is arbitrary and $\Pi$ is injective, it follows that $x_i \in\ashpp$.
	
	Given $M\in{}_\ashp\M$, consider the following linear map: 
	\[
		\begin{alignedat}1
			\mu_M\colon A\otimes_\ah M\Hp	& \longrightarrow M\\
			a\otimes m 						& \longmapsto (a\Sp 1_H) \cdot m,
		\end{alignedat}
	\]which is clearly well defined.

  Initially let us see that the following property holds:
	if $a\Sp h\in \ashp, b\in A$ and $m\in M\Hp$, then
	\begin{equation}\label{mmn-1}
		(a\Sp h)(b\Sp 1_H)\cdot m = (a(h\Ape b)\Sp 1_H)\cdot m = \Pi(a\Sp h)(b)\Sp 1_H\cdot m
	\end{equation}
	In fact,
	\begin{eqnarray*}
		(a\Sp h)(b\Sp 1_H)\cdot m
			&=&a(h_1\Ape b)\Sp h_2\cdot m\\
			&=&(a(h_1\Ape b)\Sp 1_H)(1_A\Sp h_2)\cdot m\\
			&=&(a(h_1\Ape b)\Sp 1_H)\cdot((1_A\Sp h_2)\cdot m)\\
			&=&(a(h_1\Ape b)\Sp 1_H)\cdot(((h_2\Ape 1_A)\Sp 1_H) \cdot m)\\
			&=&(a(h_1\Ape b)\Sp 1_H)((h_2\Ape 1_A)\Sp 1_H) \cdot m\\
			&=&((a(h_1\Ape b)(h_2\Ape 1_A))\Sp 1_H) \cdot m\\
			&=& (a(h\Ape b)\Sp 1_H)\cdot m\\
			&=& \Pi(a\Sp h)(b)\Sp 1_H\cdot m.
	\end{eqnarray*}
	
	Since $A$ is \ashp-module, it follows that $A\otimes_\ah M\Hp$ also is. Thus we have that
	\begin{eqnarray*}
		\mu_M((a\Sp h)\cdot (b\otimes m))
			&=&\mu_M(((a\Sp h)\cdot b)\otimes m)\\
			&=&\mu_M(a(h\Ape b)\otimes m)\\
			&=&(a(h\Ape b)\Sp 1_H)\cdot m\\
			&=&(a\Sp 1_H)\cdot[((h\Ape b)\Sp 1_H)\cdot m]\\
			&\overset{\eqref{mmn-1}}=&(a\Sp 1_H)\cdot[(1\Sp h)(b\Sp 1_H)\cdot m]\\
			&=&[(a\Sp 1_H)(1\Sp h)(b\Sp 1_H)]\cdot m\\
			&=&[(a\Sp h)(b\Sp 1_H)]\cdot m\\
			&=&(a\Sp h)\cdot[(b\Sp 1_H)\cdot m]\\
			&=&(a\Sp h)\cdot\mu_M(b\otimes m),
	\end{eqnarray*}
	Therefore $\mu_M$ is a homomorphism of \ashp-modules.
	
	To see that $\mu\ud{M}$ is bijective, consider the following map:
	\[
		\begin{alignedat}1
			\nu\ud{M}\colon M	& \longrightarrow A\otimes_\ah M\Hp\\
					m	& \longmapsto \sum_i a_i \otimes x_i\cdot m,
		\end{alignedat}
	\]which is well defined because given $h\in H$,
	\begin{eqnarray*}
		(1_A \Sp h)\cdot(x_i \cdot m)
			&=& ((1_A \Sp h) x_i) \cdot m\\ x_i \in \ashpp \to
			&=& (((h\Ape 1_A) \Sp 1_H) x_i) \cdot m\\
			&=& ((h\Ape 1_A) \Sp 1_H) \cdot (x_i \cdot m).
	\end{eqnarray*}
	
	Our goal is to show that $\nu$ and $\mu$ are mutually inverse. Although, before that, we need to prove the following equality:
	\begin{equation}\label{mmn-4}
		\sum_i (a_i\Sp 1_H) x_i = 1_A \Sp 1_H.
	\end{equation}
	Indeed, given $a\in A$, we have that
	\begin{eqnarray*}
		\Pi(\sum_i (a_i\Sp 1_H) x_i)(a)
			&=&\sum_i a_i\Pi(x_i)(a)\\
			&=&\sum_i a_i f_i(a)\\
			&=& a\\
			&=& \Pi(1_A\Sp 1_H)(a).
	\end{eqnarray*}Since $a$ is arbitrary and $\Pi$ is injective, the equality holds.
	
	Now we finally show that $\nu\ud{M}$ is the inverse of $\mu\ud{M}$. Indeed,
	\begin{eqnarray*}
		\mu(\nu(m))
			&=&\mu(\sum_i a_i \otimes x_i\cdot m)\\
			&=&\sum_i (a_i\Sp 1_H) \cdot (x_i\cdot m)\\
			&=&\sum_i ((a_i\Sp 1_H) x_i)\cdot m\\
			&\overset{\eqref{mmn-4}}=& 1_A\Sp 1_H\cdot m\\
			&=& m.
	\end{eqnarray*}
	
	By the other side,
	\begin{eqnarray*}
		\nu(\mu(a\otimes m))
			&=& \nu((a\Sp 1_H) \cdot m)\\
			&=& \sum_i a_i \otimes (x_i \cdot ((a\Sp 1_H) \cdot m))\\
			&=& \sum_i a_i \otimes (x_i (a\Sp 1_H) \cdot m)\\
			&\overset{\eqref{mmn-1}}=& \sum_i a_i \otimes \Pi(x_i)(a)\Sp 1_H \cdot m\\
			&=& \sum_i a_i \otimes f_i(a)\Sp 1_H \cdot m\\
			&=& \sum_i a_i f_i(a) \otimes 1_A\Sp 1_H \cdot m\\
			&=& a \otimes m,
	\end{eqnarray*}which concludes the proof.
.
\end{proof}

From ring theory, we have the following result. 

\begin{teo}\label{lema-5}\cite{W}*{Theorem~18.8}
	Let $R$ be a unital associative ring with identity $1_R$, $M$ be a left $R$-module and consider $S = \text{End}({}_R M)$. Then the following statements are equivalent:
	\begin{enumerate}[\indent (i)]
		\item $M$ is a generator in ${}_R \M$;
		\item $M$ a finitely generated projective left $S$-module and  $\varphi\colon R \to \text{End}({}_S M)$ given by $\varphi(r)(m) = r \cdot m$ is an isomorphism of rings and of $R$-modules.\qed
	\end{enumerate}
\end{teo}

The theorem above shall be adapted to our context. For this, we initially prove that  $\ahop \simeq \text{End}({}_{A\Sp H} A)$ as rings, where $\ahop$ is the opposite algebra of $\ah$. 

\begin{prop}\label{lema-6}
	Let $A$ be a symmetric partial $H$-module algebra. Then $\ahop$ is isomorphic to $End(_{\ashp} A)$ as ring.
\end{prop}

\begin{proof}
	Consider the linear map
	\[
		\begin{alignedat}1
			\varphi\colon \ahop	& \longrightarrow {\rm End}\,(_{\ashp} A)\\
			a					& \longmapsto \varphi(a)\colon b \longmapsto R_a(b)= ba
		\end{alignedat}
	\]
	that is well defined because
	\begin{eqnarray*}
		\varphi(a)(c\Sp h \Acao b)
			&=& c(h \Ape b)a\\
			&=& c(h \Ape ba) \\
			&=& {c\Sp h} \Acao ba \\
			&=& {c\Sp h} \Acao \varphi(a)(b).
	\end{eqnarray*}
	
	Therefore, $\varphi(a) \in End(_{\ashp}A)$ for all $a \in \ahop$. Furthermore, $\varphi$ is injective because $A$ has identity. It only remains us to show that $\varphi$ is surjective. For this, take $f \in End(_{\ashp}A)$ and note that $f(1_A) \in \ah$. In fact, given $h\in H$, we have that
	\begin{eqnarray*}
		h \cdot f(1_A)
			&=& {1_A \Sp h} \Acao f(1_A) \\
			&=& f({1_A \Sp h} \Acao 1_A) \\
			&=& f(h \Ape 1_A) \\
			&=& f({h\Ape 1_A \Sp 1_H} \Acao 1_A)\\
			&=& {h\Ape 1_A \Sp 1_H} \Acao f(1_A) \\
			&=& (h\Ape 1_A) f(1_A).
	\end{eqnarray*}
	Moreover, $\varphi(f(1_A)) = f$ and then $\varphi$ is surjective.
	
We also have that $\varphi$ is a homomorphism of algebras. Clearly, $\varphi(1_A) = I_A$ and, given $a,b,c \in A$, we have that
	\begin{eqnarray*}
		\varphi(a \cdot_{op} b)(c)
			&=& c (a \cdot_{op} b)\\
			&=& c\,b\, a\\
			&=& \varphi(a) (\varphi(b) (c)),
	\end{eqnarray*}
	which ends the proof.
\end{proof}

\begin{obs}
	We know that $A$ is a left $\ashp$-module via $(a\Sp h) \Acao b = a (h\Ape b)$ and also a left $\ahop$-module via $x \Acao  a = a\, x$. Moreover, $A$ is a $(\ah, \ashp)$-bimodule. Indeed, if $x \in \ahop$, $a \Sp h\in \ashp$ and $b \in A$, then
	\begin{align*}
		x \Acao ((a \Sp h) \Acao b)
			&= x \Acao (a(h \Ape b))\\
			&= a(h \Ape b)x\\ x\in\ah \to
			&= a(h \Ape bx)\\
			&= a(h \Ape (x \Acao b))\\
			&= (a\Sp h) \Acao (x \Acao b).
	\end{align*}
\end{obs}

In Theorem \ref{lema-5}, take $\ashp$ as the ring $R$, $A$ as the module $M$, and $\ahop$ as the ring $S = {\rm End}\,({}_R M)$. Therefore we obtain the following result.

\begin{teo}\label{observ}
	With the same notations above, the following statements are equivalent:
	\begin{enumerate}[(1)]
		\item $A$ is finitely generated projective left $\ah$-module  and  $\Pi: \ashp \to End(A_{\ah})$ given by $\Pi({a \Sp h})(b) =  {a \Sp h} \Acao b$ is an isomorphism of  rings and of left $\ashp$-modules;\label{observ-1}
		\item $A$ is a generator in the category $_{\ashp} \M$.\label{observ-2}\qed
	\end{enumerate}\relax
\end{teo}


So we can finally state the main result of this work.

\begin{teo}[Galois~Equivalences]\label{galois}\label{teo-pHG}
	Let $A$ be a symmetric partial $H$-module algebra and $0 \neq t \in \int_{l}^H$. The following statements are equivalent:
	\begin{enumerate}\Not{pHG}
		\item $A$ is $H^\ast$-Galois ($\Can\colon ~ A \otimes_{\ah} A \to A \Tr H^\ast$ is surjective);\label{teo-pHG-1}
		\item $[~,~]\colon A \otimes_{\ah} A \to \ashp$ given by $[~,~](a \otimes b) = atb$ is surjective;\label{teo-pHG-2}
		\item There exist $\{x_i\}_{i=1}^n$, $\{y_i\}_{i=1}^n$ in $A$ such that for all $h \in H$ we have that $T(h)1_H = \sum\limits_{i=1}^n x_i(h \cdot y_i)$ for $T \in \int_{r}^{H^\ast}$;\label{teo-pHG-3}
		\item $A$ is a generator in the category $_{\ashp} \M$;\label{teo-pHG-5}
		\item $A$ is a finitely generated projective right $\ah$-module and $\Pi: \ashp \to End(A_{\ah})$ given by $\Pi({a \Sp h})(b) =  {a \Sp h} \Acao b$ is an isomorphism of  rings and of left $\ashp$-modules;\label{teo-pHG-4}
		\item $A$ is a finitely generated projective right $\ah$-module and for all $M \in _{\ashp}\M$ the map $\mu\ud{M}\colon A \otimes_{\ah} M\Hp \to M$ given by $\mu\ud{M}(a \otimes_{\ah} m) = a \Acao m$ is an isomorphism of left $\ashp$-modules;\label{teo-pHG-6}
		\item $\mu\colon A \otimes_\ah (\ashpp) \to \ashp$ given by $\mu(a \otimes_{\ah} x) = (a \Sp 1) x$ is surjective.\label{teo-pHG-7}
	\end{enumerate}
\end{teo}
\begin{proof}
	
	(\eqref{teo-pHG-1} $\iff$ \eqref{teo-pHG-2}) It follows by Proposition \ref{cancolchete}.

	(\eqref{teo-pHG-1} $\implies$ \eqref{teo-pHG-3}) Since the canonical map is surjective, take $\{x_i\}_{i=1}^n,\ \{y_i\}_{i=1}^n$ such that $\Can(\sum\limits_{i=1}^n x_i\otimes y_i) = 1\Tr T = 1^0 \otimes T\ast 1^1.$
	
	Thus, consider the following linear map:
	\begin{eqnarray*}
	\varphi\colon A\otimes H^\ast &\longrightarrow& Hom(A\otimes H, A)\\
	a\otimes f& \longmapsto& \varphi(a\otimes f) (b\otimes h) = ab f(h).
	\end{eqnarray*}
	Therefore, we have that
	\begin{eqnarray*}
	1_A T(h)
	&=& 1^0 \varepsilon_{H^\ast}(1^1)T(h)\\
	&=& 1^0 (T\ast 1^1)(h)\\
	&=& \varphi (1_0\otimes T\ast 1^1)(1_A\otimes h)\\
	&=& \varphi (\Can(\sum\limits_{i=1}^n x_i\otimes y_i))(1_A\otimes h)\\
	&=& \varphi (\sum\limits_{i=1}^n x_i {y_i}_0\otimes {y_i}_1)(1_A\otimes h)\\
	&=& \sum\limits_{i=1}^n x_i {y_i}_0 {y_i}_1(h)\\
	&=& \sum\limits_{i=1}^n x_i (h \cdot y_i)
	\end{eqnarray*}

	(\eqref{teo-pHG-3} $\implies$ \eqref{teo-pHG-2}) Remember that, considering $t\neq0$ a left integral in $H$ and $T\neq0$ a right integral in $H^\ast$ such that $T(t)=1$, we have that $T(h_1)h_2 = T(h)1_H$ for all $h$ in $H$ and, moreover,  $T(S(h)t_1)t_2 = h$.

	Now, given $a\Sp \in\ashp$, we have that
\begin{eqnarray*}
	\sum\limits_{i=1}^n(a(h \cdot x_i) \Sp t)(y_i \# 1)
		&=& \sum\limits_{i=1}^n a(h \cdot x_i)(t_1 \cdot y_i) \Sp t_2\\
		&=& \sum\limits_{i=1}^n a(h_1 \cdot (x_i (S(h_2)t_1 \cdot y_i))) \Sp t_2\\
		&=& a(h_1 \cdot T(S(h_2)t_1) 1_A) \Sp t_2\\
		&=& a(h_1 \cdot 1_A) \Sp T(S(h_2)t_1)t_2\\
		&=& a(h_1 \cdot 1_A) \Sp h_2 \\
		&=& a \Sp h.
\end{eqnarray*}
	
	Thus $[~,~](\sum\limits_{i=1}^n(a(h \cdot x_i) \Sp t)\otimes_{\ah}(y_i \Sp 1)) = a \Sp h$, that is, $[~,~]$ is surjective.
		
	(\eqref{teo-pHG-2} $\implies$ \eqref{teo-pHG-5}) It follows from Lemma \ref{lema-3}.
	
	(\eqref{teo-pHG-5} $\iff$ \eqref{teo-pHG-4}) It follows from Theorem \ref{observ}.

	(\eqref{teo-pHG-4} $\implies$ \eqref{teo-pHG-6}) It follows from Lemma \ref{lema-4}.

	(\eqref{teo-pHG-6} $\implies$ \eqref{teo-pHG-7}) It is enough to consider $M=\ashp$.

	(\eqref{teo-pHG-7} $\implies$ \eqref{teo-pHG-2}) Recall that the map
	\begin{eqnarray*}
		\alpha\colon A & \longrightarrow & \ashpp\\
		a & \longmapsto & (1_A\Sp t)(a\Sp 1_H),
	\end{eqnarray*} is an isomorfism of $\Bbbk$-vector spaces and notice that, clearly, it is left \ah-linear. Therefore, since $(1_A \Sp t)(A\Sp 1_H) = \ashpp$ by \eqref{aigualdade}, we have that $I\otimes_\ah\alpha\colon A\otimes_\ah A \to A\otimes_\ah (\ashpp)$  is surjective.
	
	Furthermore, it is straightforward to see that $[~,~] = \mu\circ(I\otimes_\ah \alpha)$. Since, by hypothesis, $\mu$ is surjective, then [~,~] is also surjective.
\end{proof}

\bibliographystyle{abbrv}

\begin{thebibliography}{99}

\bib{AB}{article}{
   author={Alves, Marcelo Muniz S.},
   author={Batista, Eliezer},
   title={Partial Hopf actions, partial invariants and a Morita context},
   journal={Algebra Discrete Math.},
   date={2009},
   number={3},
   pages={1--19},
   issn={1726-3255},
   review={\MR{2640384}},
}

\bib{BV}{article}{
   author={Batista, Eliezer},
   author={Vercruysse, Joost},
   title={Dual constructions for partial actions of Hopf algebras}, journal={J. Pure Appl. Algebra},
   volume={220},
   date={2016},
   number={2},
   pages={518--559},
   issn={0022-4049},
   review={\MR{3399377}},
   doi={10.1016/j.jpaa.2015.05.014},
}

\bib{CJ}{article}{
   author={Caenepeel, Stefaan},
   author={Janssen, Kris},
   title={Partial (co)actions of Hopf algebras and partial Hopf-Galois
   theory},
   journal={Comm. Algebra},
   volume={36},
   date={2008},
   number={8},
   pages={2923--2946},
   issn={0092-7872},
   review={\MR{2440292}},
   doi={10.1080/00927870802110334},
}



\bib{CPQS}{article}{
   author={Castro, Felipe},
   author={Paques, Antonio},
   author={Quadros, Glauber},
   author={Sant'Ana, Alveri},
   title={Partial actions of weak Hopf algebras: smash product,
   globalization and Morita theory},
   journal={J. Pure Appl. Algebra},
   volume={219},
   date={2015},
   number={12},
   pages={5511--5538},
   issn={0022-4049},
   review={\MR{3390037}},
   doi={10.1016/j.jpaa.2015.05.031},
}
	
\bib{CQ}{article}{
   author={Castro, Felipe},
   author={Quadros, Glauber},
   title={Globalizations for partial (co)actions on coalgebras},
   journal={Algebra Discrete Math.},
   volume={27},
   date={2019},
   number={2},
   pages={212--242},
   issn={1726-3255},
   review={\MR{3982304}},
}

\bib{L}{article}{
   author={Lomp, Christian},
   title={Integrals in Hopf algebras over rings},
   journal={Comm. Algebra},
   volume={32},
   date={2004},
   number={12},
   pages={4687--4711},
   issn={0092-7872},
   review={\MR{2111109}},
   doi={10.1081/AGB-200036837},
}

	\bib{M}{book}{
   author={Montgomery, Susan},
   title={Hopf algebras and their actions on rings},
   series={CBMS Regional Conference Series in Mathematics},
   volume={82},
   publisher={Published for the Conference Board of the Mathematical
   Sciences, Washington, DC; by the American Mathematical Society,
   Providence, RI},
   date={1993},
   pages={xiv+238},
   isbn={0-8218-0738-2},
   review={\MR{1243637}},
   doi={10.1090/cbms/082},
}
	
	\bib{W}{book}{
   author={Wisbauer, Robert},
   title={Foundations of module and ring theory},
   series={Algebra, Logic and Applications},
   volume={3},
   edition={Revised and translated from the 1988 German edition},
   note={A handbook for study and research},
   publisher={Gordon and Breach Science Publishers, Philadelphia, PA},
   date={1991},
   pages={xii+606},
   isbn={2-88124-805-5},
   review={\MR{1144522}},
}
	
\end{thebibliography}

\end{document}